\documentclass[12pt]{article}
\usepackage{amsmath,amsthm,amscd,amssymb}
\usepackage[mathscr]{eucal}
\usepackage{amssymb}
\usepackage{latexsym}

\theoremstyle{definition}
\newtheorem{Def}{Definition}[section]
\newtheorem{Thm}[Def]{Theorem}

\newtheorem{Rem}[Def]{Remark}

\newtheorem{Cor}[Def]{Corollary}
\newtheorem{Prob}[Def]{Problem}
\newtheorem{Lem}[Def]{Lemma}

\numberwithin{equation}{section}

\begin{document}
\title{Ramanujan type congruences for the Klingen-Eisenstein series}
\author{Toshiyuki Kikuta and Sho Takemori}
\maketitle


\begin{abstract}
In the case of Siegel modular forms of degree $n$, we prove that, for almost all prime ideals $\frak{p}$ in any ring of algebraic integers, mod $\frak{p}^m$ cusp forms are congruent to true cusp forms of the same weight. As an application of this property, we give congruences for the Klingen-Eisenstein series and cusp forms, which can be regarded as a generalization of Ramanujan's congruence. We will conclude by giving numerical examples.
\end{abstract}

\section{Introduction}
Kurokawa \cite{Kuro1} found some examples of congruence relations on eigenvalues between the Klingen-Eisenstein series and Hecke eigen cusp forms, in the case of Siegel modular forms of degree $2$. Mizumoto \cite{Mizu4} and Katsurada-Mizumoto \cite{Kat-Miz} showed some congruence properties of this kind for more general cases. In this paper, we prove congruences on \textit{Fourier} \textit{coefficients} between the Klingen-Eisenstein series and cusp forms, in the case of Siegel modular forms of degree $n$. We remark that congruences on Fourier coefficients are stronger properties than congruences on eigenvalues of eigen forms.

In order to show these congruences, we determine all mod $\frak{p}^m$ cusp forms which are congruent to true cusp forms, where ``mod $\frak{p}^m$ cusp forms'' are Siegel modular forms  of degree $n$ whose Fourier coefficients of rank $r$ with $0\le r\le n-1$ vanish modulo $\frak{p}^m$ (see Definition \ref{cusp}). Namely, we can explain our main results as follows: \\
(1) In the case of Siegel modular forms of degree $n$, for almost all prime ideals $\frak{p}$ in any ring of algebraic integers, mod $\frak{p}^m$ cusp forms are congruent to true cusp forms of the same weight (Theorem \ref{Lem3}). \\
(2) We take a prime ideal $\frak{p}$ such that a constant multiple of the Klingen-Eisenstein series $\alpha [f]_r^n$ attached to a Hecke eigen cusp form $f$ is a mod $\frak{p}^m$ cusp form. Then there exists a cusp form $F$ such that $\alpha [f]_r^n\equiv F$ mod $\frak{p}^m$ (Corollary \ref{ThmM}).

The congruences we prove can be regarded as a generalization of Ramanujan's congruence which asserts that
\begin{align*}
\sigma _{11}(n)\equiv \tau (n) \bmod{691},
\end{align*}
where $\sigma _m(n)$ is the $n$-th Fourier coefficient of the Eisenstein series of weight $12$ (i.e., the sum of $m$-th powers of the divisors of $n$)
and $\tau (n)$ is the $n$-th Fourier coefficient of Ramanujan's
$\Delta $ function. In the case of degree $2$ and of $f=1$ for the situation (2), we already proved these congruences in \cite{Ki-Na}.

\section{Preliminaries}
\label{Pre}
\subsection{Notation}
First we confirm the notation. For the elementally facts, we refer to Klingen \cite{Kli}. Let $\Gamma _n=Sp_n(\mathbb{Z})$ be the Siegel modular group of degree $n$ and $\mathbb{H}_n$
the Siegel upper-half space of degree $n$. We denote by
$M_k(\Gamma _n)$ the $\mathbb{C}$-vector space of all
Siegel modular forms of weight $k$ for $\Gamma_n$, and $S_k(\Gamma _n)$ is the subspace of cusp forms.

Any $f(Z)$ in $M_k(\Gamma _n)$ has a Fourier expansion of the form
\[
f(Z)=\sum_{0\le T\in \Lambda _n}a(T;F)q^T,\quad q^T:=e^{2\pi i\text{tr}(TZ)},
\quad Z\in\mathbb{H}_n,
\]
where $T$ runs over all elements of $\Lambda _n$, and
\begin{align*}
\Lambda_n&:=\{ T=(t_{ij})\in Sym_n(\mathbb{Q})\;|\; t_{ii},\;2t_{ij}\in\mathbb{Z}\; \}.
\end{align*}
For a subring $R$ of $\mathbb{C}$, let $M_{k}(\Gamma _n)_{R}\subset M_{k}(\Gamma _n)$ denote the $R$-module of
all modular forms whose Fourier coefficients lie in $R$.

Let $r$ be a non-negative integer with $0\le r\le n-1$. Let $\Delta _{n,r}$ be the (Klingen) parabolic subgroup of $\Gamma _n$ defined by
\[\Delta _{n,r}:=\left\{\begin{pmatrix}* & * \\ 0_{n-r,n+r} & *\end{pmatrix}\in \Gamma _n \right\}.\]
Let $k$ a positive even integer with $k>n+r+1$ and $f\in S_k(\Gamma _r)$ a Hecke eigen form. Then the Klingen-Eisenstein series attached to $f$ is defined by
$$[f]_r^n(Z):=\sum_{
M=\left(\begin{smallmatrix} A & B\\
C & D\end{smallmatrix}\right)\in \Delta _{n,r} \backslash \Gamma_n}
\det(CZ+D)^{-k} f((MZ)^*) \qquad (Z\in {\mathbb H}_n);$$
here $Z^*$ denotes the $r\times r$-submatrix in the upper left corner of
$Z$. This series $[f]_r^n$ defines a Hecke eigen form which belongs to $M_k(\Gamma _n)$. Let $K_f$ be the number field generated over $\mathbb{Q}$ by the eigenvalues of the Hecke operators over $\mathbb{Q}$ on $f$. Then it is known that $[f]_r^n\in M_k(\Gamma _n)_{K_f}$ by \cite{Kuro4,Kuro5,Mizu5}.

Let $\Phi :M_k(\Gamma _n)\rightarrow M_k(\Gamma _{n-1})$ be the Siegel $\Phi$-operator. Then we have
\begin{align}
\label{Phi}
\Phi ([f]_r^n)=
\begin{cases}
[f]_r^{n-1}\ &{\rm if}\ n>r+1,\\
f\ &{\rm if}\ n=r+1.
\end{cases}
\end{align}

\section{Main results and their proofs}
\subsection{Main results}
Let $K$ be an algebraic number field and ${\mathcal O}={\mathcal O}_K$ the ring of integers in $K$. For a prime ideal $\frak{p}$ in ${\mathcal O}$, we denote by ${\mathcal O}_{\frak{p}}$ the localization of ${\mathcal O}$ at $\frak{p}$. First our main result concerns ``mod $\frak{p}^m$ cusp forms'' defined as
\begin{Def}
\label{cusp}
Let $f\in M_k(\Gamma _n)_{{\mathcal O}_{\frak{p}}}$. We call $f$ a $mod$ $\frak{p}^m$ $cusp$ $form$ if $\Phi (f)\equiv 0$ mod $\frak{p}^m$.
\end{Def}
\begin{Thm}
\label{Lem3}
For a finite set $S_n(K)$ of prime ideals in $K$ depends on $n$, we have the following: Let $k>2n$ and $\frak{p}$ be a prime ideal of ${\mathcal O}$ with $\frak{p}\not \in S_n(K)$. Let $f\in M_k(\Gamma _n)_{{\mathcal O}_{\frak{p}}}$ be a mod $\frak{p}^m$ cusp form. In other words, we assume that $f\in M_k(\Gamma _n)_{{\mathcal O}_{\frak{p}}}$ satisfies $\Phi (f)\equiv 0$ mod $\frak{p}^m$. Then there exists $g\in S_k(\Gamma _n)_{{\mathcal O}_{\frak{p}}}$ such that $f\equiv g$ mod $\frak{p}^m$.
\end{Thm}
\begin{Rem}
\label{k:even}
Since there does not exist non-cusp form of odd weight, the statement for the case where $k$ is odd in Theorem \ref{Lem3} is trivial.
\end{Rem}
We will see how to determine the exceptional set $S_n(K)$ in the later section (Definition \ref{S_n}).
As an application of this theorem, we obtain congruences between the Klingen-Eisenstein series and cusp forms:

Let $v_\frak{p}$ be the normalized additive valuation with respect to $\frak{p}$. We define two values $v_\frak{p}(f)$ and $v_\frak{p}^{(n')}(f)$ for $f\in M_k(\Gamma _n)_K$ by
\begin{align*}
&v_\frak{p}(f):=\min \{v_\frak{p}(a_f(T))\;|\;T \in \Lambda _{n} \}, \\
&v^{(n')}_\frak{p}(f):=\min \{v_\frak{p}(a_f(T))\;|\;T \in \Lambda _{n},\ {\rm rank}(T)=n' \}\quad (0\le n'\le n).
\end{align*}
Then we have
\begin{Cor}
\label{ThmM}
Let $k>2n$ be even and $f\in S_k(\Gamma _r)_{K_f}$ ($n>r$) a Hecke eigen form. For the Klingen-Eisenstein series $[f]_r^n$ attached to $f$, we choose a prime ideal $\frak{p}$ in ${\mathcal O}_{K_f}$ with $\frak{p}\not \in S_n(K_f)$ such that $v_\frak{p}^{(n)}([f]_r^n)= v_\frak{p}(\Phi ([f]_r^{n}))-m$ ($m\in \mathbb{Z}_{\ge 1}$). Then there exists $F\in S_k(\Gamma _n)_{{\mathcal O}_{\frak{p}}}$ such that $\alpha [f]_r^n\equiv F$ mod $\frak{p}^m$ for some $0\neq \alpha \in \frak{p}^{m}$.
\end{Cor}
\begin{Rem}
  (1) The assumption $v_\frak{p}^{(n)}([f]_r^n)=v_\frak{p}(\Phi([f]_r^{n}))-m$ is equivalent to the fact that $\alpha [f]_r^n$ is a non-zero mod $\frak{p}^m$ cusp form for some $\alpha \in \frak{p}^m$ satisfying $v_{\frak{p}}(\alpha [f]_r^n)=0$. \\
  (2) For a prime $l$ and $1 \le i \le n$, we define Hecke operators $T(l)$ and
  $T_{i}(l^{2})$ by $T(l) = \Gamma_{n}\mathrm{diag}(1_{n}, l 1_{n})\Gamma_{n}$
  and $T_{i}(l^{2}) = \Gamma_{n}\mathrm{diag}(1_{i}, l 1_{n-i}, l^{2} 1_{i}, l
  1_{n-i})\Gamma_{n}$.  For an eigen form $F$ and a Hecke operator $T$, we
  denote by $\lambda(T, F)$ the Hecke eigenvalue of $T$.  By Deligne-Serre
  lifting lemma (\cite{D-S} Lemma 6.11), we can take an eigen form $G\in S_k(\Gamma _n)$
  such that $\lambda(T, [f]_r^n) \equiv \lambda(T, G)
  \bmod{\frak{p}}$ for $T = T(l),\
  T_{i}(l^{2})$, $l \ne p$  and $1 \le i \le n$.\\
  (3) If $r=0$, $[f]_r ^n$ is the ordinary Siegel-Eisenstein series. In
  particular, if $n=2$, this was proved by \cite{Ki-Na}.
\end{Rem}
Using the integrality theorem obtained by Mizumoto \cite{Mizu2}, we can give conditions on $\frak{p}$ to find congruences for the Klingen-Eisenstein series and cusp forms as in Corollary \ref{ThmM}. We shall introduce an example:

To apply his theorem, we assume that \\
~~~(i) $f\in M_k(\Gamma _r)_{O_{K_f}}$ and one of the Fourier coefficients of $f$ is equal to $1$,\\
~~~(ii) $L(k-r,f,{\rm St})\neq 0$, where $L(s,f,{\rm St})$ is the standard $L$-function of $f$. \\
Then Mizumoto's result states that
\[a(T;[f]_r^n)\in c_k(r,n)\mu _k(r)^{-1}\prod _{i=r+1}^{\left[ \frac{n+r}{2} \right] } {\rm Num} \left( \frac{B_{2k-2i}}{k-i} \right)^{-1} \cdot L^*(k-r,f,{\rm St})^{-1}{\mathcal A}(f)^{-1} \]
for some $c_k(r,n)\in \mathbb{Q}^{\times}$ and $\mu _k(r)\in \mathbb{Z}$ which are computable. Here ${\mathcal A}(f)$ is an integral ideal of $O_{K_f}$, ${\rm Num}(*)$ is the numerator,
\begin{align*}
&L^*(k-r,f,{\rm St}):=\frac{L(k-r,f,{\rm St})}{\pi ^{(2r+1)k-\frac{3r(r+1)}{2}}(f,f)}\in K_f
\end{align*}
and $(f,f)$ is the Petersson norm of $f$. For the precise definitions of these numbers, see \cite{Mizu2}. This property tells us all possible primes appearing in denominators of all Fourier coefficients of $[f]_r^n$, since the property (\ref{Phi}).

For example, we consider a simple case where $r=n-1$. We choose $\frak{p}$ satisfying
\[v_{\frak{p}}\left(
c_k(r,n)\mu _k(r)^{-1}\prod _{i=r+1}^{\left[ \frac{n+r}{2} \right]} {\rm Num} \left( \frac{B_{2k-2i}}{k-i} \right)^{-1} \cdot L^*(k-r,f,{\rm St})^{-1}\right) =-m.
\]
Then $\alpha [f]_r^n$ is a mod $\frak{p}^m$ cusp form for any $\alpha \in \frak{p}^m$. Applying Theorem \ref{Lem3}, we can find $F\in S_k(\Gamma _n)_{{\mathcal O}_{\frak{p}}}$ such that $\alpha [f]_r^n\equiv F$ mod $\frak{p}^m$. Remark that it may become $\alpha [f]_r^n\equiv F\equiv 0$ mod $\frak{p}^m$ for this choice of $\frak{p}$, compared with Corollary \ref{ThmM}.

\subsection{Proof of the theorem}
In order to define $S_n(K)$ and to prove the theorem, we start with introducing some basic properties.
\begin{Lem}
\label{Lem6}
Let $\bigoplus_{k} M_k(\Gamma _{n})_{\mathbb{Z}_{(p)}}=\mathbb{Z}_{(p)}[f_1,\cdots,f_{s}]/C$ with a relation $C$ among the generators. Then we have $\bigoplus_{k} M_k(\Gamma _{n})_{{\mathcal O}_{\frak{p}}}={\mathcal O}_{\frak{p}}[f_1,\cdots,f_{s}]/C$.
\end{Lem}
\begin{proof}
  By the same argument of Mizumoto \cite{Mizu3} Lemma A.4, we have
  $M_{k}(\Gamma_{n})_{\mathbb{Z}_{(p)}}\otimes_{\mathbb{Z}_{(p)}}
  {\mathcal O}_{\frak{p}} = M_k(\Gamma _{n})_{{\mathcal O}_{\frak{p}}}$.
  Thus we have
  \begin{align*}
    \bigoplus_{k} M_k(\Gamma _{n})_{{\mathcal O}_{\frak{p}}}
    &=
    \left(
      \bigoplus_{k} M_k(\Gamma _{n})_{\mathbb{Z}_{(p)}}
    \right)
    \otimes_{\mathbb{Z}_{(p)}}{\mathcal O}_{\frak{p}}\\
    &=\left(
      \mathbb{Z}_{(p)}[f_1,\cdots,f_{s}]/C
    \right) \otimes_{\mathbb{Z}_{(p)}}{\mathcal O}_{\frak{p}}
    = {\mathcal O}_{\frak{p}}[f_1,\cdots,f_{s}]/C.
  \end{align*}
\end{proof}

The finite generation of $\bigoplus_{k} M_k(\Gamma _{n})_{\mathbb{Z}}$ is known by Faltings-Chai \cite{F-C}. Namely, we always assume that $\bigoplus_{k} M_k(\Gamma _{n})_{\mathbb{Z}_{(p)}}=\mathbb{Z}_{(p)}[f_1,\cdots,f_{s}]/C$ for any prime $p$ and hence also that $\bigoplus_{k} M_k(\Gamma _{n})_{{\mathcal O}_{\frak{p}}}={\mathcal O}_{\frak{p}}[f_1,\cdots,f_{s}]/C$ for any prime ideal $\frak{p}$.
\begin{Lem}
\label{Lem4}
Assume that $\bigoplus_{k} M_k(\Gamma _{n})_{{\mathcal O}_{\frak{p}}}={\mathcal O}_{\frak{p}}[f_1,\cdots,f_{s}]/C$ with $f_i\in M_{k_i}(\Gamma _n)_{{\mathcal O}_{\frak{p}}}$. Let $M$ be a natural number. We take the minimum of integers $\alpha _i\in \mathbb{Z}_{\ge 0}$ such that, the weight of $f_i^{\alpha _i}$ is strictly greater than $M$. Then the graded algebra $\bigoplus_{M<k} M_k(\Gamma _{n})_{{\mathcal O}_{\frak{p}}}$ is generated over ${\mathcal O}_{\frak{p}}$ by the following finitely many monomials;
\begin{align}
&f_1^{\alpha _1},\cdots,f_s^{\alpha _s}, \label{1} \\
&f_1^{i_1}\cdots f_s^{i_s}\quad (i_1k_1+\cdots +i_sk_s>M,\ 0\le i_j< 2\alpha _j) \label{2}.
\end{align}
\end{Lem}
\begin{proof}
First, we remark that any $g\in M_{k}(\Gamma _n)_{{\mathcal O}_{\frak{p}}}$ can be written by a liner combination of monomials of the form $f_1^{a_1}\cdots f_s^{a_s}$. Hence we may consider only the case $g=f_1^{a_1}\cdots f_s^{a_s}$.

Let $k_0:=\alpha _1k_1+\cdots +\alpha _sk_s$. If $2k_0\ge k>M$, then the assertion is trivial. Hence, we assume that $k>2k_0$. Now we consider $a_i=\alpha _iq_i+r_i$ ($0\le r_i<\alpha _i$). Then there exists $j_0$ such that $q_{j_0}\ge 1$ because of $k>2k_0$. In this case, we may consider the following decomposition;
\begin{align*}
&g=h_1\cdot h_2,\\
&h_1:=f_1^{r_1}\cdots f_{j_0-1}^{r_{j_0-1}}f_{j_0}^{r_{j_0}+\alpha _{j_0}}f_{j_0+1}^{r_{j_0+1}}\cdots f_{s}^{r_{s}},\\
&h_2:=f_1^{\alpha _1q_1}\cdots f_{j_0-1}^{\alpha _{j_0-1}q_{j_0-1}}f_{j_0}^{\alpha _{j_0}(q_{j_0}-1)}f_{j_0+1}^{\alpha _{j_0+1}q_{j_0+1}}\cdots f_{s}^{\alpha _sq_s}.
\end{align*}
Then, both $h_1$ and $h_2$ are written by the monomials of (\ref{1}) and (\ref{2}). This completes the proof.
\end{proof}
\begin{Lem}
\label{Lem1}
For $k>2n$, the restricted Siegel $\Phi$-operator $\Phi _K:M_k(\Gamma _n)_{K}\rightarrow M_k(\Gamma _{n-1})_{K}$ is surjective.
\end{Lem}
\begin{proof}
  By Shimura \cite{Shim}, we have $M_k(\Gamma _n)_{K}=M_k(\Gamma
  _n)_{\mathbb{Q}}\otimes_{\mathbb{Q}} K$.  Since $\mathbb{C}$ is faithfully
  flat over $K$, the surjectivity of $\Phi_{K}$ is equivalent to that of $\Phi:
  M_k(\Gamma _n)_{\mathbb{C}}\rightarrow M_k(\Gamma _{n-1})_{\mathbb{C}}$.  The
  surjectivity of $\Phi$ was proved by Klingen \cite{Kli}.  Therefore, we obtain
  the assertion of the lemma.
\end{proof}
In order to prove the theorem, it suffices to consider the case where the weight is even (see Remark \ref{k:even}). From Lemma \ref{Lem4}, we may assume that $\bigoplus_{2n<k\in 2\mathbb{Z}} M_k(\Gamma _{n-1})_{{\mathcal O}_{\frak{p}}}={\mathcal O}_{\frak{p}}[f_1,\cdots,f_{s}]/C$. Applying Lemma \ref{Lem1}, we have $\Phi ^{-1}_K(f_i)\neq \phi $ for any $i$ with $1\le i\le s$. 

We are now in a position to define the set $S_n(K)$ and to prove Theorem \ref{Lem3}. 
\begin{Def}
\label{S_n}
Let $S_n(K)$ be the set of all prime ideals $\frak{p}$ in ${\mathcal O}$ such that, there exists $i$ which satisfies that for all $F_i\in \Phi _K^{-1}(f_i)$ we have $v_{\frak{p}}(F_i)<0$. Note that $S_n(K)$ is a finite set depends on $n$ not depends on generators of $\bigoplus _{2n<k\in 2\mathbb{Z}}M_{k}(\Gamma _{n-1})_{{\mathcal O}_{\frak{p}}}$ (Remark \ref{Rem4} in Subsection \ref{REM}).
\end{Def}
\begin{proof}[Proof of Theorem \ref{Lem3}]
We choose a polynomial $P\in {\mathcal O}_{\frak{p}}[x_1,\cdots ,x_{s}]$ such that $\Phi (F)=P(f_1,\cdots ,f_s)$. Since $\Phi (F)=P(f_1,\cdots ,f_s)\equiv 0$ mod $\frak{p}^m$, there exists $\gamma \in \frak{p}^m$ such that $\displaystyle \gamma ^{-1} \Phi (F)\in M_k(\Gamma _{n-1})_{{\mathcal O}_{\frak{p}}}$. In fact, we may choose $\gamma $ as $\gamma :=a(T_0;\Phi(f))$ for some $T_0$ which satisfies $v_\frak{p}(a(T_0;\Phi (f)))=v_\frak{p}(\Phi (f))$. Hence we can find $Q\in {\mathcal O}_{\frak{p}}[x_1,\cdots ,x_{s}]$ such that $\displaystyle \gamma ^{-1} \Phi (F)=Q(f_1,\cdots ,f_s)$. Since $\frak{p}\not \in S_n(K)$, there exists $F_i\in \Phi _K^{-1}(f_i)$ such that $v_{\frak{p}}(F_i)\ge 0$ for each $i$ with $1\le i \le s$. Then $F\in \gamma Q(F_1,\cdots,F_s)+{\rm Ker} \Phi$. Hence there exists $G\in {\rm Ker}\Phi$ such that $F=\gamma Q(F_1,\cdots,F_s)+G$. Note that $Q(F_1,\cdots,F_s)\in M_k(\Gamma _n)_{{\mathcal O}_{\frak{p}}}$ because of $v_{\frak{p}}(F_i)\ge 0$ and hence $G\in S_k(\Gamma _n)_{{\mathcal O}_{\frak{p}}}$. This implies $F\equiv G$ mod $\frak{p}^m$ because of $\gamma \in \frak{p}^m$. This completes the proof of Theorem \ref{Lem3}.
\end{proof}
\subsection{Remark on ${\boldsymbol S_n(K)}$}
\label{REM}
\begin{Rem}
\label{Rem4}
For each prime ideal $\frak{p}$, it does not depend on the choice of generators of $\bigoplus_{2n<k\in 2\mathbb{Z}}M_k(\Gamma _{n-1})_{{\mathcal O}_{\frak{p}}}$ whether $\frak{p}$ belongs to the exceptional set $S_n(K)$ or not. Namely, we get the following property:
Assume that $\bigoplus _{2n<k\in 2\mathbb{Z}}M_k(\Gamma _{n-1})_{{\mathcal O}_{\frak{p}}}={\mathcal O}_{\frak{p}}[f_1,\cdots,f_s]/C={\mathcal O}_{\frak{p}}[f_1',\cdots,f_t']/C'$. If we can take $F_i\in M_k(\Gamma _{n})_{{\mathcal O}_{\frak{p}}}$ such that $\Phi (F_i)=f_i$ for each $i$ with $1\le i\le s$, then we can take $F'_j\in M_k(\Gamma _n)_{{\mathcal O}_{\frak{p}}}$ such that $\Phi (F'_j)=f'_j$ for each $j$ with $1\le j\le t$.
\end{Rem}
\begin{proof}
For each $1\le j\le t$, we can write as $f'_j=P(f_1,\cdots,f_s)$ for some polynomial $P\in {\mathcal O}_{\frak{p}}[x_1,\dots,x_s]$. If we put $F_j':=P(F_1,\cdots,F_s)$, then $F'_j\in M_k(\Gamma _n)_{{\mathcal O}_{\frak{p}}}$ and $\Phi (F'_j)=f'_j$.
\end{proof}
\begin{Rem}
\label{Rem5}
(1) We have $S_n(K)\subset \{\frak{p}\;|\; \frak{p}\cap \mathbb{Z}\in S_n(\mathbb{Q}) \}$. Hence, to obtain the congruences as in Corollary \ref{ThmM}, it suffices to except the prime ideals above $p$ with $(p)\in S_n(\mathbb{Q})$. \\
(2) Let $g:=[K:\mathbb{Q}]<\infty $. Then $S_n(K)\supset \{\frak{p}\;|\; \frak{p}\cap \mathbb{Z}=(p)\in S_n(\mathbb{Q}),\ p\nmid g\}$.
\end{Rem}
\begin{proof}
(1) Let $\frak{p}\in S_n(K)$. If we assume that $\bigoplus_{k} M_k(\Gamma _{n})_{\mathbb{Z}_{(p)}}=\mathbb{Z}_{(p)}[f_1,\cdots,f_{s}]/C$, then $\bigoplus_{k} M_k(\Gamma _{n})_{{\mathcal O}_{\frak{p}}}={\mathcal O}_{\frak{p}}[f_1,\cdots,f_{s}]/C$ by Lemma \ref{Lem6}. Since $\frak{p}\in S_n(K)$, there exists $i$ with $1\le i \le s$ such that for all $F_i\in \Phi _K^{-1}(f_i)$, we have $v_{\frak{p}}(F_i)<0$. In particular, for all $F_i\in \Phi_{\mathbb{Q}}^{-1}(f_i)$, we have $v_{\frak{p}}(F_i)<0$. \\
(2) Let $\frak{p}\cap \mathbb{Z}=(p)\in S_n(\mathbb{Q})$, $p\nmid g$ and
$\bigoplus_{k} M_k(\Gamma
_{n})_{\mathbb{Z}_{(p)}}=\mathbb{Z}_{(p)}[f_1,\cdots,f_{s}]/C$. Seeking a
contradiction, we suppose that, for each $i$ with $1\le i\le s$, there exists $F_i\in \Phi_K^{-1}(f_i)$ such that
$v_{\frak{p}}(F_i)\ge 0$. We consider $G_i:=\sum _{\sigma
  \in {\rm Emb}(K,\mathbb{C})}F_i^{\sigma }\in M_k(\Gamma _n)_{\mathbb{Q}}$.
Note that $G_i\in M_{k}(\Gamma _n)_{\mathbb{Z}_{(p)}}$ because of $v_{\frak{p}}(F_i)\ge 0$ and that $\Phi (G_i)=gf_i$ since $\Phi (F_i^{\sigma })=f_i\in M_{k}(\Gamma _n)_{\mathbb{Q}}$ for any $\sigma \in {\rm Emb}(K,\mathbb{C})$. By the assumption $p\nmid g$, we have $v_{p}(g^{-1}G_i)\ge 0$ and $g^{-1}G_i\in \Phi _{\mathbb{Q}}^{-1}(f_i)$. This contradicts for $\frak{p}\cap \mathbb{Z}=(p)\in S_n(\mathbb{Q})$.
\end{proof}
We have $S_2(\mathbb{Q})\subset \{2,3\}$ by \cite{Ki-Na}. We shall consider $S_3(\mathbb{Q})$. Let $E^{(n)}_k\in M_k(\Gamma _n)_{\mathbb{Q}}$ be the normalized Siegel-Eisenstein series of weight $k$ and degree $n$. Let $X_k\in S_k(\Gamma _2)_{\mathbb{Z}}$ ($k=10$, $12$) be Igusa's cusp forms normalized as $a\left(\left(\begin{smallmatrix}1 & 1/2 \\ 1/2 & 1\end{smallmatrix}\right);X_k\right)=1$ in \cite{Igu}. Then $\bigoplus _{k\in 2\mathbb{Z}}M_k(\Gamma _2)_{\mathbb{Z}_{(p)}}=\mathbb{Z}_{(p)}[E_4^{(2)}, E_6^{(2)}, X_{10},X_{12}]$ holds for any prime $p\ge 5$ (cf. Nagaoka \cite{Na}). Note that $E_k^{(3)}\in \Phi ^{-1}_{\mathbb{Q}}(E_{k}^{(2)})$ for any even $k$. We can construct $F_{k}\in \Phi ^{-1}_{\mathbb{Q}}(X_{k})$ ($k=10$, $12$) by 
\begin{align}
\label{F_k}
\begin{split}
F_{10}:&=-\frac{43867}{2^{10}\cdot 3^5\cdot 5^2\cdot 7\cdot 53}(E^{(3)}_{10}-E^{(3)}_4E^{(3)}_6), \\
F_{12}:&=\frac{131\cdot 593}{2^{11}\cdot 3^6\cdot 5^3\cdot 7^2\cdot 337}
(3^2\cdot 7^2E_4^{(3)3}+2\cdot 5^3E_6^{(3)2}-691E^{(3)}_{12}).
\end{split}
\end{align}
Moreover, we know all possible primes which appear in the denominators of $E_k^{(3)}$ by B\"ocherer's results \cite{Bo}. Hence, it suffices to except all primes in the denominators of the constant factors in (\ref{F_k}) and all possible primes appearing the denominators of $E_k^{(3)}$ for $k=4$, $6$, $10$, $12$. In this way, we get 
\[S_3(\mathbb{Q})\subset \{2,3,5,7,53,131,337,593,43867\}\]
\begin{Prob}
For the general degree cases, give an explicit bound $C_n$ such that
\[\max {S_n}(\mathbb{Q}) <C_n.\]
\end{Prob}
\section{Numerical examples}
We give some numerical examples of Corollary \ref{ThmM} for the case of degree $2$. For simplicity, we put $E_k:=E_k^{(2)}$. Let $\Delta \in S_{12}(\Gamma _1)$ be Ramanujan's delta function. We write simply
$(m,r,n)$ for $\left( \begin{smallmatrix}n & \frac{r}{2} \\ \frac{r}{2} & m \end{smallmatrix} \right) \in \Lambda _2$. In the following construction of examples, we apply Sturm type theorem obtained by \cite{C-C-K}. In order to prove a congruence between two modular forms of even weight $k$ of degree $2$ by using the theorem in \cite{C-C-K}, it suffices to check the congruences for Fourier coefficients for
\begin{align*}
T=&(1,0,1),\ (1,1,1)\quad {\rm if}\ 10\le k \le 18,\\
T=&(1, 0, 1),\ (1, 0, 2),\ (1, 1, 1),\ (1, 1, 2),\ (2, 0, 2),\\
&(2, 1, 2),\ (2, 2, 2)\quad {\rm if}\ 20 \le k\le 28.
\end{align*}
The reason is that all Fourier coefficients corresponding to $(n, r, m)$, $(m, r, n)$, $(n, -r, m)$, $(m, -r, n)$ are the same in the case of even weight.

\subsubsection*{Weight 12}
We consider a Hecke eigen form $f_{12}:=7\Delta \in S_{12}(\Gamma _1)$. Then the Klingen-Eisenstein series $[f_{12}]_1^2$ is a mod $7$ cusp form. Hence, there exists a cusp form $F_{12}\in S_{12}(\Gamma _2)$ such that $[f_{12}]_1^2\equiv F_{12}$ mod $7$ by Corollary \ref{ThmM}. In fact, we can confirm this congruence as follows: We set $F_{12}:=X_{12}\in S_{12}(\Gamma _2)$. The following table is of the Fourier coefficients modulo $7$ of $[f_{12}]_1^2$ and $F_{12}$:
\begin{center}
\begin{tabular}{c|c|c|c}
$T=(m, r, n)$ & $a(T;[f_{12}]_{1}^{2})$ & $a(T;F_{12})$ & \text{modulo} $7$\\
\hline
 $(1, 0, 1)$ & $1242$ & $10$ & $3$ \\
 $(1, 1, 1)$ & $92$ & $1$ & $1$
\end{tabular}
\end{center}
Applying Sturm type theorem mentioned above, we have $[f_{12}]_1^2\equiv F_{12}$ mod $7$.

\subsubsection*{Weight 16}
Let $a$ be a root of the polynomial $x^2 - x - 12837$ and put $K=\mathbb{Q}(a)$. Since $\dim S_{16}(\Gamma _1)=1$, we can find a unique cusp form $f_{16}\in S_{16}(\Gamma _1)$ such that $a(1; f_{16}) = 7^{2} \cdot 11$. If we put $\frak{p}=(7,a+4)$, then $[f_{16}]_1^2$ is a mod $\frak{p}^2$ cusp form. There exists a unique normalized Hecke eigen form $g_{30}\in S_{30}(\Gamma _1)$ such that the eigenvalue is $ -192a + 4416$ for the Hecke operator $T(2)$. Let $F_{16}\in S_{16}(\Gamma _2)$ be the Saito-Kuorokawa lift of $g_{30}$ normalized as the table below. Then we have $[f_{16}]_1^2\equiv F_{16}$ mod $\frak{p}^2$. In fact, their Fourier coefficients are given in the following table:
\begin{center}
\begin{tabular}{c|c|c|c}
$T = (m, r, n)$ & $a(T; [f_{16}]_{1}^{2})$ & $a(T;F_{16})$ & \text{modulo} $\frak{p}^{2}$\\
\hline
$(1, 0, 1)$ & $5394$ & $80a + 3600$ & $4$ \\
$(1, 1, 1)$ & $124$ & $8a + 1248$ & $26$
\end{tabular}
\end{center}
Applying Sturm type theorem repeatedly, we have $[f_{16}]_1^2\equiv F_{16}$ mod $\frak{p}^2$.

\subsubsection*{Weight 20}
In this case also $\dim S_{20}(\Gamma _1)=1$. Thus there exists a unique cusp form $f_{20}\in S_{20}(\Gamma _1)$ such that $a(1; f_{20}) = 11\cdot 71^2$. Then $[f_{20}]_1^2$ is a mod $71^2$ cusp form. Let $F_{20}\in S_{20}(\Gamma _2)$ be the unique Hecke eigen form such that $F_{20}$ is not Saito-Kurokawa lift. Explicitly, we can write as
\[F_{20} = 38 (E_{4}E_{6}X_{10} + E_{4}^2X_{12} - 1785600X_{10}^2).\]
Then we have $[f_{20}]_1^2\equiv F_{20}$ mod $71^2$. In fact, we can confirm this by the following table and an application of Sturm type theorem:
\begin{center}
\begin{tabular}{c|c|c|c}
$T = (m, r, n)$ & $a(T; [f_{20}]_{1}^{2})$ & $a(T;F_{20})$ & \text{modulo} $71^{2}$\\
\hline
$(1, 0, 1)$ & $10386$ & $304$ & $304$ \\
$(1, 0, 2)$ & $1925356716$ & $198816$ & $2217$ \\
$(1, 1, 1)$ & $76$ & $76$ & $76$ \\
$(1, 1, 2)$ & $162929376$ & $4256$ & $4256$ \\
$(2, 0, 2)$ & $1238800286736$ & $-335343616$ & $3868$ \\
$(2, 1, 2)$ & $385264596000$ & $278989920$ & $816$ \\
$(2, 2, 2)$ & $9084897120$ & $-63912960$ & $1879$
\end{tabular}
\end{center}
\subsubsection*{Weight 22}
Since $\dim S_{22}(\Gamma _1)=1$, there exists a unique cusp form $f_{22}\in S_{22}(\Gamma _1)$ such that $a(1;f_{22}) = 7 \cdot 13 \cdot 17 \cdot 61 \cdot 103$. Then $[f_{22}]_1^2$ is a mod $61$ cusp form. Let $F_{22}\in S_{22}(\Gamma _2)$ be the unique Hecke eigen form such that $F_{22}$ is not Saito-Kurokawa lift. Explicitly, we can write as
\[F_{22} = 2\cdot 3^{-2}(-61E_{4}^3X_{10} - 5E_{6}^2X_{10} + 30E_{4}E_{6}X_{12} + 80870400X_{10}X_{12}).\]
Then we have $[f_{22}]_1^2\equiv F_{22}$ mod $61$. In fact, we can confirm this by the following table and Sturm type theorem:
\begin{center}
\begin{tabular}{c|c|c|c}
$T = (m, r, n)$ & $a(T; [f_{22}]_{1}^{2})$ & $a(T;F_{22})$ & \text{modulo} $61$ \\
\hline
$(1, 0, 1)$ & $-179610$ & $96$ & $35$ \\
$(1, 0, 2)$ & $-133169475780$ & $-1728$ & $41$ \\
$(1, 1, 1)$ & $-740$ & $-8$ & $53$ \\
$(1, 1, 2)$ & $-8620265280$ & $-10752$ & $45$ \\
$(2, 0, 2)$ & $54428790246720$ & $-313368576$ & $14$ \\
$(2, 1, 2)$ & $15093047985984$ & $142287360$ & $41$ \\
$(2, 2, 2)$ & $223472730240$ & $17725440$ & $60$
\end{tabular}
\end{center}

\section*{Acknowledgment}
The authors would like to thank Professor S.~Nagaoka and Professor S. B\"ocherer for the valuable discussions about the proofs. The authors would also like to thank Professor H. Katsurada for his informing them about the value of congruences on Fourier coefficients between Klingen-Eisenstein series and cusp forms.


\begin{flushleft}
Toshiyuki Kikuta\\
College of Science and Engineering\\
Ritsumeikan University\\
1-1-1 Noji-higashi, Kusatsu\\
Shiga 525-8577, Japan\\
E-mail: kikuta84@gmail.com\\

\vspace{0.5cm}

\noindent
Sho Takemori\\
Department of Mathematics,\\
Kyoto University\\
Kitashirakawa-Oiwake-Cho, Sakyo-Ku, \\
Kyoto, 606-8502, Japan\\
E-mail: takemori@math.kyoto-u.ac.jp
\end{flushleft}
\end{document}